\documentclass[12pt]{amsart}
\usepackage{amsmath, amsthm, amssymb}

\usepackage{xcolor}
\usepackage{float}
\usepackage{mathtools}
\usepackage{cite}
\usepackage{stix}
\usepackage{setspace}
\usepackage{adjustbox}
\usepackage{tabularx}
\usepackage{subcaption}
\usepackage{braket}
\usepackage{tikzit}
\usepackage{nicematrix}
\usepackage{blindtext}
\usepackage[breaklinks]{hyperref}
\usepackage{bookmark}
\usepackage[toc,page]{appendix}
\usepackage[font={small}]{caption}
\usepackage{tikzcode}

\newcommand{\bmb}{\left( \begin{array}{rr}}
\newcommand{\enm}{\end{array}\right)}
\DeclareTextFontCommand{\emph}{\bfseries\em}

\newcommand{\cM}{\mathcal M}

\hypersetup{
    colorlinks,
    linkcolor={red!50!black},
    citecolor={blue!100!black},
    urlcolor={blue!80!black}
}

\newcommand{\Z}{{\mathbb Z}}

\newcommand{\R}{{\mathbb R}}

\newcommand{\oV}{{\overline{V}}}

\numberwithin{equation}{section}

\theoremstyle{definition}
\newtheorem{thm}{Theorem}[section]
\newtheorem{prop}[thm]{Proposition}

\newtheorem{lemma}[thm]{Lemma}

\newtheorem{cor}[thm]{Corollary}
\newtheorem{remark}[thm]{Remark}

\begin{document}

\title{Decomposing the Action of Macdonald Operators for Classical Types}
\author{Hieu Trung Vu}
\address{Yau Mathematical Sciences Center, Tsinghua University, Beijing, China}
\email{htrungvu@mail.tsinghua.edu.cn}

\begin{abstract}
We study Macdonald operators associated to classical root systems of types $B_N$, $C_N$, and $D_N$, continuing the work in \cite{DV26}. We provide explicit formulas for decomposing the cumulative action of these operators in terms of sum of the actions of lower-rank operators and auxiliary operators. We also discuss the same idea in the elliptic case, also known as Ruijsenaars operators, and show that under the trigonometric limit, the identity reduces to type-$A_N$ decomposition.
\end{abstract}

\maketitle

\section{Introduction}
The Macdonald operators, also known as Macdonald-Ruijsenaars operators, are a family of commuting difference operators that act on the space of symmetric functions. Originally, they were defined as the following: 
\begin{equation}
    \label{eq:macdonald_operator}
    \cM^r_N(x_1, \ldots, x_N) = \sum_{\substack{I \subset [N] \\ |I| = r}} \prod_{i \in I, j \notin I} \frac{t x_i - x_j}{x_i - x_j} \Gamma_{I},
\end{equation}
where $[N]:=\{1, \dots, N\}$, $\displaystyle \Gamma_{I} = \prod_{i \in I} \Gamma_i$ and $\Gamma_i$ is the $q$-shift operator defined by
\[
(\Gamma_i f)(x_1, \ldots, x_N) = f(x_1, \ldots, q x_i, \ldots, x_N),
\]
with complex parameters $q$ and $t$. They were found around the same time by Ruijsenaars in \cite{Ruijsenaars1987} as the trigonometric limit of a larger family of operators, the Hamiltonians of the Ruijsenaars-Schneider model. The operators of the form \eqref{eq:macdonald_operator} are referred to as degree-$r$ Macdonald operators associated with the root system $A_{N-1}$. In \cite{Macdonald2000}, Macdonald defined the Macdonald operators associated to other root systems. These operators form a family of commuting difference operators acting diagonally on the Macdonald symmetric functions. In this note, we will consider an extension of a recent result on the rank recursion for Macdonald operators type $A_{N-1}$ to Macdonald operators of classical types \cite{DV26}. Upon fixing the root data as in Table \ref{table:root-weight-data}, the Macdonald operators associated to classical root systems of types $B_N$, $C_N$, and $D_N$ take the form \eqref{eq:type D_1}-\eqref{eq:type_CN_macdo}. These explicit formulas showed up in \cite{difrancesco2021macdonald} and can also be derived from Weyl orbit summation in \cite{Macdonald2000}, which we include the details in the appendix for completion.

\begin{align}
    \label{eq:type D_1}
    D_N: \quad \mathcal E_1^{\left(D_N\right)}&=\sum_{\epsilon= \pm 1} \sum_{i=1}^N \prod_{j \neq i} \frac{1-t x_i^\epsilon x_j}{1-x_i^\epsilon x_j} \frac{t x_i^\epsilon-x_j}{x_i^\epsilon-x_j} \Gamma_i^{2 \epsilon}\\
    \label{eq:type D_N-1}
    \quad \mathcal{E}_{N-1}^{\left(D_N\right)}&=\sum_{\substack{\epsilon_1, \ldots, \epsilon_N= \pm 1 \\ \epsilon_1 \epsilon_2 \cdots \epsilon_N=-1}} \prod_{1 \leq i<j \leq N} \frac{1-t x_i^{\epsilon_i} x_j^{\epsilon_j}}{1-x_i^{\epsilon_i} x_j^{\epsilon_j}} \prod_{i=1}^N \Gamma_i^{\epsilon_i}\\
    \label{eq:type D_N}
    \mathcal{E}_N^{\left(D_N\right)}&=\sum_{\substack{\epsilon_1, \ldots, \epsilon_N= \pm 1 \\ \epsilon_1 \epsilon_2 \cdots \epsilon_N=1}} \prod_{1 \leq i<j \leq N} \frac{1-t x_i^{\epsilon_i} x_j^{\epsilon_j}}{1-x_i^{\epsilon_i} x_j^{\epsilon_j}} \prod_{i=1}^N \Gamma_i^{\epsilon_i}\\
    \label{eq:type_BN_macdo}
    B_N: \quad \mathcal{E}_1^{\left(B_N\right)}&=\sum_{\epsilon= \pm 1} \sum_{i=1}^N \frac{1-t x_i^\epsilon}{1-x_i^\epsilon} \prod_{j \neq i} \frac{1-t x_i^\epsilon x_j}{1-x_i^\epsilon x_j} \frac{t x_i^\epsilon-x_j}{x_i^\epsilon-x_j} \Gamma_i^{2 \epsilon},\\
    \label{eq:type_CN_macdo}
    C_N: \quad \mathcal E_N^{\left(C_N\right)}&=\sum_{\epsilon_1, \ldots, \epsilon_N= \pm 1} \prod_{i=1}^N \frac{1-t x_i^{2 \epsilon_i}}{1-x_i^{2 \epsilon_i}} \prod_{1 \leq i<j \leq N} \frac{1-t x_i^{\epsilon_i} x_j^{\epsilon_j}}{1-x_i^{\epsilon_i} x_j^{\epsilon_j}} \prod_{i=1}^N \Gamma_i^{\epsilon_i}.
\end{align}
In \cite{DV26}, we were able to write $\mathcal M^r_n(x_1,\ldots,x_N)$ as a sum of lower-rank Macdonald operators, providing a recursive structure for type $A_{N-1}$ Macdonald operators. One should understand this result as the cumulative action of the Macdonald operator on a function of $N$-variables being expressed as the sum of action of lower rank Macdonald operators to the same function. Motivated by the same idea, we decompose the cumulative action of all classical types Macdonald operators  \eqref{eq:type D_1}-\eqref{eq:type_CN_macdo}. For these cases, the decomposition involves not only lower-rank Macdonald operators associated to the same root system but also certain operators that we called auxiliary operators, which do not preserve the ring of symmetric polynomials. The three-term decomposition structure we obtain from type-$A_N$ serves as a guiding principle for the analogous decompositions of $\mathcal E^{(B_N)}_1$ and $\mathcal E^{(D_N)}_1$. The three-term recurrence relation of type $A_N$ at degree-1 can be extended to the elliptic case, known as the Ruijsenaars operators defined in \cite{Ruijsenaars1987}.

Our main result is the following relations: 
\begin{align*}
    \mathcal E_1^{(B_N)}&= w(x_p^{-1},x_p)\,\mathcal E^B\!\big(\mathcal S\setminus\{(p,+1)\}\big)  + w(x_p,x_p^{-1})\,\mathcal E^B\!\big(\mathcal S\setminus\{(p,-1)\}\big)  - t\,\mathcal E^{(B_{N-1})}_1\!\big(\widehat{x}_p\big)\\
    \mathcal E_N^{(C_N)}&=\sum_{\sigma=\pm1}
\left(\frac{1-tx_p^{2\sigma}}{1-x_p^{2\sigma}}\right)
\mathcal E^C\!\bigl([N]\setminus\{p\};\{(p,\sigma)\}\bigr)\,\Gamma_p^{\sigma},\\
 \mathcal E_1^{(D_N)}&= w(x_p^{-1},x_p)\,\mathcal E^D\!\big(\mathcal S\setminus\{(p,+1)\}\big)  + w(x_p,x_p^{-1})\,\mathcal E^D\!\big(\mathcal S\setminus\{(p,-1)\}\big)  - t\,\mathcal E^{(D_{N-1})}\!\big(\widehat{x}_p\big).
\end{align*}
where $\mathcal E^B(S), \mathcal E^C(S), \mathcal E^D(S)$, for some index subset $S \subset \{1, \dots,n\}$, are the auxiliary operators defined in \eqref{eq:typeB_aux}, \eqref{eq: type C auxiliary op} and \eqref{eq:typeD_aux} with some prefactors $\displaystyle w(x_i,x_j):=\frac{tx_i-x_j}{x_i-x_j}$.

\section{\texorpdfstring{$\mathcal E^{(B_N)}_1$ and $\mathcal E^{(D_N)}_1$ Decomposition}{E1 of BN and E1 of DN decomposition}}
Let $\mathcal S:=\{(i, \epsilon): \epsilon = \pm 1, 1\leq i\leq N\}$ and fix the convention that for any $\alpha = (i, \epsilon) \in \mathcal S$, $x_\alpha:=x^\epsilon_i$ and $x_0=1$. For every subset $S\subseteq\mathcal S$ define the auxiliary operator
\begin{equation}
\label{eq:typeB_aux}
\mathcal E^B(S)
:=\sum_{(i,\epsilon)\in S}
\frac{1}{w(x_i^\epsilon,x_i^{-\epsilon})}
\prod_{\substack{\beta\in S\cup\{0\}\\ \beta\neq(i,\epsilon)}}
w(x_i^\epsilon,x_{\beta})\;\Gamma_i^{2\epsilon}.
\end{equation}
The following corollary provides a different expression for the Macdonald operator of type $B_N$ in \eqref{eq:type_BN_macdo}.
\begin{cor}
    \label{cor:B_N macdo rewrite}
\begin{equation}
    \begin{aligned}
    \mathcal{E}_1^{\left(B_N\right)}
    &=\mathcal E^B(\mathcal S) =\sum_{(i, \epsilon) \in \mathcal S}
    \frac{1}{w(x_i^\epsilon,x_i^{-\epsilon})}
    \prod_{\substack{\beta \in \mathcal S \cup \{0\} \\ \beta \neq (i, \epsilon)}}
    w(x_i^\epsilon, x_\beta)\,\Gamma_i^{2\epsilon}
    \end{aligned}
\end{equation}
\end{cor}

\begin{proof}
In the product
\[
\prod_{\substack{\beta \in \mathcal S \cup \{0\} \\ \beta \neq (i, \epsilon)}} w(x_i^\epsilon, x_\beta),
\]
there are contributions of $w(x_i^\epsilon, x_i^{-\epsilon})$, $w(x_i^\epsilon, x_0)=w(x_i^\epsilon, 1)$, and $w(x_i^\epsilon,x_j)\,w(x_i^\epsilon,x_j^{-1})$ for every $j\neq i$. The term $w(x_i^\epsilon,x_i^{-\epsilon})$ cancels the prefactor $1/w(x_i^\epsilon,x_i^{-\epsilon})$, while the remaining terms give exactly the product in the original definition of the Macdonald operator of type $B_N$, $\mathcal E_1^{B_N}$.
\end{proof}

The case of degree-1 $A_{N}$ Macdonald operator, the rank recurrence relation reads: 
\begin{equation} 
    \label{eq: Level 1 type A recurrence}
\begin{aligned}
\mathcal{M}_{N+1}^1\left(x_1, \ldots, x_{N+1}\right) & =\frac{t x_N-x_{N+1}}{x_N-x_{N+1}} \mathcal{M}_N^1\left(x_1, \ldots, x_N\right)+\frac{t x_{N+1}-x_N}{x_{N+1}-x_N} \mathcal{M}_N^1\left(x_1, \ldots, x_{N-1}, x_{N+1}\right) \\
& -t \mathcal{M}_{N-1}^1\left(x_1, \ldots, x_{N-1}\right)
\end{aligned}
\end{equation}
It is more convenient to rewrite this recurrence relation in terms of $w(x_i,x_j)$. Namely,  define the operator
\[
A_{[N]}:=\sum_{i \in [N]} \prod_{\substack{j \in [N] \\ j \neq i}} w(x_i,x_j) \Gamma_i=\cM^1_N(x_1, \dots, x_N).
\]
Then \eqref{eq: Level 1 type A recurrence} can be written as 
\[
A_{[N+1]} = w(x_N,x_{N+1}) A_{[N+1]\setminus \{N+1\}} + w(x_{N+1},x_N) A_{[N+1]\setminus \{N\}} - t A_{[N-1]}.
\]
In general, one can generalize \eqref{eq: Level 1 type A recurrence} to any two indices $p,q \in [N+1]$, 
\begin{prop}
    \begin{equation}
        \label{eq: type A level 1 recur in omega}
\begin{aligned}
A_{[N+1]}
&= w(x_p,x_q) A_{[N+1]\setminus \{q\}} + w(x_q,x_p) A_{[N+1]\setminus \{p\}} - t A_{[N-1]}
\end{aligned}
\end{equation}
\end{prop}
It is also possible to extend the operator $A$ to other set of variable indices such as $\mathcal S$. Define an operator $\mathcal A$ acting on the function of variables $\{x_\beta: \beta \in S \subseteq \mathcal S\} \cup \{x_0\}$ as
\[
\mathcal A_{S}:=\sum_{\alpha \in S} \prod_{\substack{\beta \in S \cup \{0\} \\ \beta \neq \alpha}} w(x_\alpha, x_\beta) \Gamma^2_\alpha.
\]
where for $\alpha =(i,\epsilon) \in S$, we have $\Gamma^2_\alpha =\Gamma^{2\epsilon}_i$ and
\[
(\Gamma_\alpha^2 f)(x_1, \dots, x_N)
= f(x_1, \dots, x_{i-1}, q^{2\epsilon}x_i, x_{i+1}, \dots, x_N).
\]
Then, one can check that $\mathcal A_S$ satisfies the recurrence relation similar to $A_{[N]}$ for $K = \{\alpha_1, \alpha_2\}\subset S$:
\begin{prop}
\begin{equation}
    \label{eq: A_S recur}
    \mathcal A_S = w(x_{\alpha_2},x_{\alpha_1})\,\mathcal A_{S\setminus\{\alpha_1\}} + w(x_{\alpha_1},x_{\alpha_2})\,\mathcal A_{S\setminus\{\alpha_2\}} - t\,\mathcal A_{S\setminus\{\alpha_1,\alpha_2\}}.
\end{equation}
\end{prop}
\begin{proof}
    Compare the coefficient of each $q$-shift $\Gamma_\alpha$.
    If $\alpha=\alpha_1$, then only the term $A_{S \setminus\{\alpha_2\}}$ on the right contains
  $\Gamma_{\alpha_1}$. The corresponding coefficient is
  $$w(x_{\alpha_1},x_{\alpha_2})\prod_{\substack{\beta \in S\setminus\{\alpha_1,\alpha_2\}}}w(x_{\alpha_1},x_\beta)
  =\prod_{\substack{\beta \in S\setminus\{\alpha_1\}}}w(x_{\alpha_1},x_\beta),$$
  which is exactly the coefficient of $\Gamma_{\alpha_1}$ in $\mathcal A_S$. The case $\alpha=\alpha_2$ is the
  same.
    If $\alpha \notin K$, then all three terms on the right contain $\Gamma_\alpha$, and
  the coefficient is
  $$
  \begin{aligned}
  &\quad w(x_{\alpha_2},x_{\alpha_1})\prod_{\substack{\beta \in S\setminus\{\alpha,\alpha_1\}}}w(x_\alpha,x_\beta)
  +w(x_{\alpha_1},x_{\alpha_2})\prod_{\substack{\beta \in S\setminus\{\alpha,\alpha_2\}}}w(x_\alpha,x_\beta)
  -t\prod_{\substack{\beta \in S\setminus\{\alpha,\alpha_1,\alpha_2\}}}w(x_\alpha,x_\beta)\\
  &= \Big(w(x_{\alpha_2},x_{\alpha_1})w(x_{\alpha},x_{\alpha_2})+w(x_{\alpha_1},x_{\alpha_2})w(x_{\alpha},x_{\alpha_1})-t\Big)
  \prod_{\substack{\beta \in S\setminus\{\alpha,\alpha_1,\alpha_2\}}}w(x_\alpha,x_\beta),
  \end{aligned}
  $$
  which is the coefficient of $\Gamma_\alpha$ in $\mathcal A_S$.
\end{proof}

\begin{lemma}
    \label{lemma: general set recursion}
    Let $K = \{\alpha_1, \alpha_2\} \subset S$. Then, $\mathcal E^B(S)$ satisfies the same recurrence relation as $\mathcal A_S$: 
    $$\displaystyle
\mathcal E^B(S)
=
w(x_{\alpha_2},x_{\alpha_1})\,\mathcal E^B(S\setminus\{\alpha_1\})
+
w(x_{\alpha_1},x_{\alpha_2})\,\mathcal E^B(S\setminus\{\alpha_2\})
-\;t\,\mathcal E^B(S\setminus\{\alpha_1,\alpha_2\}).
$$        
\end{lemma}
\begin{proof}
    The operators $\mathcal E^B(S)$ and $\mathcal A_S$ are only different by terms involving the null variable $x_0$ and the prefactor $\displaystyle \frac{1}{w(x_i^\epsilon, x_i^{-\epsilon})}$. Passing to the recurrence relation, these differences do not affect the coefficients of the shifts $\Gamma_\alpha$ for $\alpha \in S$, and hence $\mathcal E^B(S)$ satisfies the same recurrence relation as $\mathcal A_S$.
\end{proof}
The following theorem is then immediate from the Lemma \ref{lemma: general set recursion} by taking $S=\mathcal S$ and $K=\{(p,+1),(p,-1)\}$ for all $p \in [N]$.

\begin{thm}
    \label{thm: Rank recursion for type B macdo}
    Type-$B_N$ Macdonald operator satisfies the rank recursion relation: 
    {\footnotesize
    \begin{equation}
        \label{eq: type B rank recurrence}
        \begin{aligned}
        \mathcal E_1^{(B_N)}
        &= w(x_p^{-1},x_p)\,\mathcal E^B\!\big(\mathcal S\setminus\{(p,+1)\}\big) + w(x_p,x_p^{-1})\,\mathcal E^B\!\big(\mathcal S\setminus\{(p,-1)\}\big) - t\,\mathcal E^B\!\big(\mathcal S\setminus\{(p,+1),(p,-1)\}\big) \\
        &= w(x_p^{-1},x_p)\,\mathcal E^B\!\big(\mathcal S\setminus\{(p,+1)\}\big)  + w(x_p,x_p^{-1})\,\mathcal E^B\!\big(\mathcal S\setminus\{(p,-1)\}\big)  - t\,\mathcal E^{B_{N-1}}\!\big(\widehat{x}_p\big),
        \end{aligned}
    \end{equation}
    }
    where $\widehat{x}_p = (x_1, \dots, x_{p-1}, x_{p+1}, \dots, x_N).$
\end{thm}

\begin{remark}
    From Theorem \ref{thm: Rank recursion for type B macdo}, the rank recursion for $\mathcal E_1^{(D_N)}$ is then straightforward: 
    {\footnotesize
    \begin{equation}
        \label{eq: type D rank recurrence}
        \begin{aligned}
        \mathcal E_1^{(D_N)}
        &= w(x_p^{-1},x_p)\,\mathcal E^D\!\big(\mathcal S\setminus\{(p,+1)\}\big) + w(x_p,x_p^{-1})\,\mathcal E^D\!\big(\mathcal S\setminus\{(p,-1)\}\big)- t\,\mathcal E^D\!\big(\mathcal S\setminus\{(p,+1),(p,-1)\}\big) \\
        &= w(x_p^{-1},x_p)\,\mathcal E^D\!\big(\mathcal S\setminus\{(p,+1)\}\big)  + w(x_p,x_p^{-1})\,\mathcal E^D\!\big(\mathcal S\setminus\{(p,-1)\}\big)  - t\,\mathcal E^{(D_{N-1})}\!\big(\widehat{x}_p\big),
        \end{aligned}
    \end{equation}
    }
    where: 
\begin{equation}
    \label{eq:typeD_aux}
    \mathcal E^D(S) =\sum_{(i,\epsilon)\in S}
\frac{1}{w(x_i^\epsilon,x_i^{-\epsilon})}
\prod_{\substack{\beta\in S\\ \beta\neq(i,\epsilon)}}
w(x_i^\epsilon,x_{\beta})\;\Gamma_i^{2\epsilon}.
\end{equation}
The only difference between type $B$ and type $D$ in this context is the presence of the null variable $x_0$ in type $B$ coming from the short root, which is absent in type $D$.
\end{remark}
\section{\texorpdfstring{$\mathcal E^{(C_N)}_N$, $\mathcal E^{(D_N)}_N$ and $\mathcal E^{(D_N)}_{N-1}$ Decomposition}{EN of CN, EN of DN and E N-minus-1 of DN decomposition}}
\subsection{\texorpdfstring{The Case of $\mathcal E^{(C_N)}_N$}{The case of EN of CN}}
For a finite set $V \subseteq [N]$, define the \emph{signed complement set} $ \overline{V} := \{(i,\epsilon_i)\mid i \notin V, \epsilon_i \in \{\pm 1\}\}$. Then the auxiliary operator in the case of type $C$ Macdonald operator is given by
\begin{equation}
    \label{eq: type C auxiliary op}
    \begin{aligned}
    \mathcal E^C(V;\oV)
    &=\sum_{\substack{\epsilon_i, i \in V\\ \epsilon_i = \pm 1}}
    \left(\prod_{i \in V}\frac{1-tx_i^{2\epsilon_i}}{1-x_i^{2\epsilon_i}}\right)
    \left(\prod_{\substack{i<j,\\i,j \in V}}\frac{1-tx_i^{\epsilon_i}x_j^{\epsilon_j}}{1-x_i^{\epsilon_i}x_j^{\epsilon_j}}\right)
    \left(\prod_{\substack{i \in V,\\(k,\epsilon_k) \in \oV}}\frac{1-tx_i^{\epsilon_i}x_k^{\epsilon_k}}{1-x_i^{\epsilon_i}x_k^{\epsilon_k}}\right)
    \left[\prod_{i \in V}\Gamma_i^{\epsilon_i}\right]
    \end{aligned}
\end{equation}
Comparing this operator with type-$C_N$ Macdonald operator, it is a stratight-forward check that $\mathcal E_N^{(C_N)} = \mathcal E^C([N];\varnothing)$, which is nothing but removing the factor $\displaystyle \prod_{\substack{i\in V\\(k,\epsilon_k)\in \overline V}}\frac{1-tx_i^{\epsilon_i}x_k^{\epsilon_k}}{1-x_i^{\epsilon_i}x_k^{\epsilon_k}}$. 
\begin{lemma}
    \label{lemma:typeC_auxiliary_variable_removal}
    For any finite set $V\subseteq [N]$ and any $p\in V$,
% any slot collection $(\beta,V_{\beta})_{\beta\in B}$, and any $p\in V$,
% $$
% E^C\bigl(V;(\beta,V_{\beta})_{\beta\in B}\bigr)
% =
% \sum_{\sigma=\pm1}
% c(x_p^{2\sigma})
% \Bigl[\prod_{\substack{\beta\in B\\ p\in V_{\beta}}}c(x_p^{\sigma}z_{\beta})\Bigr]
% E^C\!\bigl(V\setminus\{p\};
% (\beta,V_{\beta}\setminus\{p\})_{\beta\in B}\cup((p,\sigma),V\setminus\{p\})\bigr)
% \,\Gamma_p^{\sigma}.
% $$
% In particular, for the full-coupling operator,
$$
E^C(V;\overline V)
=
\sum_{\epsilon_p=\pm1}
\frac{1-tx_p^{2\epsilon_p}}{1-x_p^{2\epsilon_p}}
\left(\prod_{\substack{i \in V\\(k,\epsilon_k) \in \overline V}} \frac{1-tx_i^{\epsilon_i}x_k^{\epsilon_k}}{1-x_i^{\epsilon_i}x_k^{\epsilon_k}}\right)
E^C\!\bigl(V\setminus\{p\};\overline V\cup\{(p,\epsilon_p)\}\bigr)\,\Gamma_p^{\epsilon_p}.
$$
\end{lemma}
\begin{proof}
Fixing some $(\epsilon_1, \ldots, \epsilon_N)\in \{\pm 1\}^N$ and considering the coefficient of the monomial $\displaystyle \prod_{i \in V}\Gamma_i^{\epsilon_i}$ on the l.h.s, we separate the component corresponding to the variable $x_p$ for any $p \in V$: 
\begin{equation}
    \label{eq: LHS_remove_typeC}
    \begin{aligned}
    L\left(\prod_{i \in V}\Gamma_i^{\epsilon_i}\right)  &=\prod_{i\in V}\frac{1-tx_i^{2\epsilon_i}}{1-x_i^{2\epsilon_i}}\prod_{\substack{i<j\\ i,j\in V}}\frac{1-tx_i^{\epsilon_i}x_j^{\epsilon_j}}{1-x_i^{\epsilon_i}x_j^{\epsilon_j}}\prod_{\substack{i \in V\\(k,\epsilon_k) \in \overline V}}\frac{1-tx_i^{\epsilon_i}x_k^{\epsilon_k}}{1-x_i^{\epsilon_i}x_k^{\epsilon_k}}\\
&=\frac{1-tx_p^{2\epsilon_p}}{1-x_p^{2\epsilon_p}}\left(\prod_{j\in V\setminus\{p\}}\frac{1-tx_p^{\epsilon_p}x_j^{\epsilon_j}}{1-x_p^{\epsilon_p}x_j^{\epsilon_j}}\right)\left(\prod_{\substack{(k,\epsilon_k) \in \overline V}} \frac{1-tx_p^{\epsilon_p}x_k^{\epsilon_k}}{1-x_p^{\epsilon_p}x_k^{\epsilon_k}}\right)L'.
    \end{aligned}
\end{equation}
where $\displaystyle L' = \prod_{i\in V\setminus\{p\}}\frac{1-tx_i^{2\epsilon_i}}{1-x_i^{2\epsilon_i}}\prod_{\substack{i<j\\ i,j\in V\setminus\{p\}}}\frac{1-tx_i^{\epsilon_i}x_j^{\epsilon_j}}{1-x_i^{\epsilon_i}x_j^{\epsilon_j}}\prod_{\substack{i \in V\setminus\{p\}\\(k,\epsilon_k) \in \overline V}}\frac{1-tx_i^{\epsilon_i}x_k^{\epsilon_k}}{1-x_i^{\epsilon_i}x_k^{\epsilon_k}}$. Moving to the r.h.s., the auxiliary operator $\mathcal E^C\!\bigl(V\setminus\{p\};\overline V\cup\{(p,\epsilon_p)\}\bigr)$ can only contain the monomial $\displaystyle \prod_{i \in V\setminus\{p\}}\Gamma_i^{\epsilon_i}$ with the corresponding coefficient
\[R'=L'\prod_{j \in V\setminus\{p\}}\frac{1-tx_p^{\epsilon_p}x_j^{\epsilon_j}}{1-x_p^{\epsilon_p}x_j^{\epsilon_j}}\]
from the added factors corresponding to $\overline V\cup\{(p,\epsilon_p)\}$. Then, combining with the prefactors $\displaystyle \frac{1-tx_p^{2\epsilon_p}}{1-x_p^{2\epsilon_p}} \left(\prod_{\substack{i \in V\\(k,\delta) \in \overline V}} \frac{1-tx_i^{\epsilon_i}x_k^{\delta}}{1-x_i^{\epsilon_i}x_k^{\delta}}\right)$ and $\Gamma_p^{\epsilon_p}$ we obtain exactly $\displaystyle L\left(\prod_{i \in V}\Gamma_i^{\epsilon_i}\right)$ and completes the proof.
\end{proof}

Applying Lemma \ref{lemma:typeC_auxiliary_variable_removal} we have the decomposition of the type-$C_N$ Macdonald operator in terms of the auxiliary operator:

\begin{thm}
\label{thm:typeC_macdo_recur}
$$
\mathcal E_N^{(C_N)}
=
\sum_{\sigma=\pm1}
\left(\frac{1-tx_p^{2\sigma}}{1-x_p^{2\sigma}}\right)
\mathcal E^C\!\bigl([N]\setminus\{p\};\{(p,\sigma)\}\bigr)\,\Gamma_p^{\sigma},
$$
\end{thm}

Let us do an example in the case of $C_2$.
$$
\mathcal E_N^{(C_2)}
=
\sum_{\epsilon_1,\epsilon_2=\pm1}
\frac{1-tx_1^{2\epsilon_1}}{1-x_1^{2\epsilon_1}}\,\frac{1-tx_2^{2\epsilon_2}}{1-x_2^{2\epsilon_2}}
\frac{1-tx_1^{\epsilon_1}x_2^{\epsilon_2}}{1-x_1^{\epsilon_1}x_2^{\epsilon_2}}
\Gamma_1^{\epsilon_1}\Gamma_2^{\epsilon_2}.
$$
Taking $p=1$, the auxiliary operator is
$$
\mathcal E^C(\{2\};\{(1,+1)\})
=\frac{1-tx_2^{2}}{1-x_2^{2}}\frac{1-tx_2x_1}{1-x_2x_1}\Gamma_2+
\frac{1-tx_2^{-2}}{1-x_2^{-2}}\,\frac{1-tx_2^{-1}x_1}{1-x_2^{-1}x_1}\Gamma_2^{-1}
$$
Then, multiplying on the left with $\displaystyle \frac{1-tx_1^2}{1-x_1^2}$ and on the right by $\Gamma_1$, we get two terms: 
\[\frac{1-tx_1^2}{1-x_1^2}\frac{1-tx_2^{2}}{1-x_2^{2}}\frac{1-tx_2x_1}{1-x_2x_1}\Gamma_2\Gamma_1+
\frac{1-tx_1^2}{1-x_1^2}\frac{1-tx_2^{-2}}{1-x_2^{-2}}\,\frac{1-tx_2^{-1}x_1}{1-x_2^{-1}x_1}\Gamma_2^{-1}\Gamma_1\]
Repeating the computation with $\mathcal E^C(\{2\};\{(1,-1)\})$ gives the rest of $\mathcal E_N^{(C_2)}$.

\subsection{Type D Decomposition}
Type $D$ auxiliary operators can be defined similarly to the type $C$ case. However, defining them this way is not particularly interesting. Considering two sets 
\begin{equation}
    \begin{aligned}
    \overline V_{\textrm{odd}} &:= \{(k,\epsilon_k)\mid k \notin V, \epsilon_k \in \{\pm 1\}, \quad  \prod_{\substack{j \in V\\k \notin V}} \epsilon_j\epsilon_k=-1\}, \\
    \overline V_{\textrm{even}} &:= \{(k,\epsilon_k)\mid k \notin V, \epsilon_k \in \{\pm 1\}, \quad \prod_{\substack{j \in V\\k \notin V}} \epsilon_j\epsilon_k=1\},
    \end{aligned}
\end{equation}
then we have type $D$ auxiliary operators for each of the cases $\mathcal E^{(D_N)}_N$ and $\mathcal E^{(D_N)}_{N-1}$ are, respectively, 
\begin{equation}
        \label{eq: type D auxiliary op}
    \begin{aligned}
    \mathcal E^D_{\textrm{odd}}(V;\oV_{\textrm{odd}})
    &=\sum_{\substack{\epsilon_i, i \in V\\ \epsilon_i = \pm 1}}
    \left(\prod_{\substack{i<j,\\i,j \in V}}\frac{1-tx_i^{\epsilon_i}x_j^{\epsilon_j}}{1-x_i^{\epsilon_i}x_j^{\epsilon_j}}\right)
    \left(\prod_{\substack{i \in V,\\(k,\epsilon_k) \in \oV_{\textrm{odd}}}}\frac{1-tx_i^{\epsilon_i}x_k^{\epsilon_k}}{1-x_i^{\epsilon_i}x_k^{\epsilon_k}}\right)
    \left[\prod_{i \in V}\Gamma_i^{\epsilon_i}\right],\\
    \mathcal E^D_{\textrm{even}}(V;\oV_{\textrm{even}})
    &=\sum_{\substack{\epsilon_i, i \in V\\ \epsilon_i = \pm 1}}
    \left(\prod_{\substack{i<j,\\i,j \in V}}\frac{1-tx_i^{\epsilon_i}x_j^{\epsilon_j}}{1-x_i^{\epsilon_i}x_j^{\epsilon_j}}\right)
    \left(\prod_{\substack{i \in V,\\(k,\epsilon_k) \in \oV_{\textrm{even}}}}\frac{1-tx_i^{\epsilon_i}x_k^{\epsilon_k}}{1-x_i^{\epsilon_i}x_k^{\epsilon_k}}\right)
    \left[\prod_{i \in V}\Gamma_i^{\epsilon_i}\right].
    \end{aligned}
\end{equation}
Then, it follows that $\mathcal E^{D_N}_N = \mathcal E^D_{\textrm{even}}([N];\varnothing)$ and $\mathcal E^{D_N}_{N-1} = \mathcal E^D_{\textrm{odd}}([N];\varnothing)$. However, because of the restriction on the parity of the product $\displaystyle \prod_{i\in V}\epsilon_i$, the analog of Lemma \ref{lemma:typeC_auxiliary_variable_removal} can only have one term where $\epsilon_p$ is totally determined so we would not advance further. 

\section{An Elliptic Generalization for Type A Macdonald Operator Decomposition}
In \cite{Ruijsenaars1987}, Ruijsenaars introduced a family of commuting difference operators which acts as Hamiltonian of the Ruijsenaars-Schneider model, a relativistic generalization of the Calogero-Moser system. For our presentation, we will use the notation in \cite{MZ22}, 
\begin{equation}
    \label{eq: elliptic_Macdo}
    D^a_N=\sum_{|I|=a} \prod_{\substack{i \in I \\ j \notin I}} \phi\left(z_j-z_i\right) \prod_{i \in I} p_i, \quad k=1, \ldots, N,
\end{equation}
where $\phi(z)=\phi(\hbar,z)$ is the elliptic Kronecker function defined as: 
\begin{equation}
    \label{eq: Elliptic Kronecker}
    \phi(z, u)=\frac{\vartheta^{\prime}(0) \vartheta(z+u)}{\vartheta(z) \vartheta(u)}=\phi(u, z).
\end{equation}
and $\vartheta(z)$ is the Jacobi theta function:
\begin{equation}
    \label{eq: Jacobi theta}
    \vartheta(z)=\vartheta(z \mid \tau)=-\sum_{k \in \mathbb{Z}} \exp \left(\pi \imath \tau\left(k+\frac{1}{2}\right)^2+2 \pi \imath\left(z+\frac{1}{2}\right)\left(k+\frac{1}{2}\right)\right), \quad \operatorname{Im}(\tau)>0.
\end{equation}
The shift operator in this case is $p_i$, which acts on a function of complex variables $f(z_1, \ldots, z_N)$ as:
\begin{equation}
    p_i f(z_1, \ldots, z_N) = \exp\left(\eta\frac{\partial}{\partial z_i}\right) f(z_1, \ldots, z_N)=f(z_1, \ldots, z_i + \eta, \ldots, z_N).
\end{equation}

The three-terms relation for $\mathcal E^{(B_N)}_1$ and $\mathcal M^r_N$ turns out to have an equivalent formulation in terms of the elliptic Macdonald operators $D^a_N$. Specifically, 
\begin{thm}
    \label{thm: Elliptic 3 terms recurrence}
    For every $N\ge 1$ the following recurrence holds:
\begin{equation}
    \begin{aligned}
D^1_{N+1}(z_1,\dots,z_{N+1})
&=\phi(z_{N+1}-z_N)D^1_N(z_1,\dots,z_N) + \phi(z_N-z_{N+1})D^1_N(z_1,\dots,z_{N-1},z_{N+1}) \\
&\quad - \mathcal E^{\textrm{ell}}_N,
\end{aligned}
\end{equation}
where $\mathcal E^{\textrm{ell}}_N$ is the elliptic auxiliary operator  given by:
{\footnotesize
\[
\begin{aligned}
\mathcal E^{\textrm{ell}}_N
&=\sum_{i=1}^{N-1}\Big(
\phi(z_{N+1}-z_N)\phi(z_N-z_i)
+\phi(z_N-z_{N+1})\phi(z_{N+1}-z_i)-\phi(z_N-z_i)\phi(z_{N+1}-z_i)
\Big)\prod_{\substack{j=1\\j\neq i}}^{N-1}\phi(z_j-z_i)\,p_i,
\end{aligned}
\]
}
with the convention that  $\mathcal E^{\textrm{ell}}_1=0$.
\end{thm}
\begin{proof}
    We prove by matching the coefficient of each shift operator $p_i$ on both sides of the recurrence relation for the case $1\le i \le N-1$. The edge cases $i=N,N+1$ can be handled similarly but with slightly simpler computations. The coefficient of the shift $p_i$ on the l.h.s. is given by: 
    \begin{equation}
        \label{eq: LHS_elliptic}
        \phi(z_N-z_i)\,\phi(z_{N+1}-z_i)\prod_{\substack{j=1\\j\neq i}}^{N-1}\phi(z_j-z_i);
    \end{equation}
    On the r.h.s., for $1\le i\le N-1$, the first and second summand contributes, respectively:
\begin{equation}
    \label{eq: RHS elliptic contribution}
    \begin{aligned}
    &\phi(z_{N+1}-z_N)\phi(z_N-z_i)\prod_{\substack{j=1\\j\neq i}}^{N-1}\phi(z_j-z_i), \quad \phi(z_N-z_{N+1})\phi(z_{N+1}-z_i)\prod_{\substack{j=1\\j\neq i}}^{N-1}\phi(z_j-z_i),
    \end{aligned}
\end{equation}
with the auxiliary operators: 
{\footnotesize
\begin{equation}
    \label{eq: RHS elliptic auxiliary}
\begin{aligned}
&-\Big(\phi(z_{N+1}-z_N)\phi(z_N-z_i)
+\phi(z_N-z_{N+1})\phi(z_{N+1}-z_i) -\phi(z_{N}-z_i)\phi(z_{N+1}-z_i)\Big)
\prod_{\substack{j=1\\j\neq i}}^{N-1}\phi(z_j-z_i)
\end{aligned}
\end{equation}}
Summing all the contributions from \eqref{eq: RHS elliptic contribution} and \eqref{eq: RHS elliptic auxiliary}, we obtain exactly the l.h.s. \eqref{eq: LHS_elliptic} after factoring out $\displaystyle \prod_{\substack{j=1\\j\neq i}}^{N-1}\phi(z_j-z_i)$.
\end{proof}

\subsection{The Trigonometric Limit}
In this section, we describe in detail the trigonometric limit of the elliptic Ruijsenaars operators to the usual type-$A_{N}$ Macdonald operator. As a result, we recover the three-term rank recursion relations for the type-$A_{N}$ Macdonald operators from the elliptic case.
Taking the limit $\mathfrak p=e^{\pi \imath \tau} \to 0$, with $|\mathfrak p|<1$, the theta function \eqref{eq: Jacobi theta} becomes: 
\[
\begin{aligned}
\vartheta(z)
&= -\sum_{k \in \Z}\mathfrak p^{(k+\frac{1}{2})^2}
\exp\left\{2 \pi \imath (k+\frac{1}{2})(z+\frac{1}{2})\right\}. = -2\sum_{n =0}^\infty\mathfrak p^{(n+\frac{1}{2})^2}
\cos\left\{2 \pi (n+\frac{1}{2})(z+\frac{1}{2})\right\} \\
&= 2\mathfrak p^{\frac{1}{4}}\sin\left(\pi z\right)+ O(\mathfrak p^{\frac{9}{4}}).
\end{aligned}
\]
Then, $\displaystyle \lim_{\mathfrak p \to 0} \phi(z)= \frac{\pi\sin(\pi(z+\hbar))}{\sin(\pi z)\sin(\pi\hbar)}$. Upon setting $\displaystyle x_i = e^{2 \pi \imath z_i}$ and $t = e^{-2 \pi \imath \hbar}$, we have:
\[
\begin{aligned}
\frac{tx_i-x_j}{x_i-x_j}
&=\frac{e^{2 \pi \imath (z_i-\hbar)}-e^{2 \pi \imath z_j}}{e^{2 \pi \imath z_i}-e^{2 \pi \imath z_j}}=e^{-\pi \imath \hbar}\frac{\sin(\pi(z_j-z_i+\hbar))}{\sin(\pi(z_j-z_i))}.
\end{aligned}
\]
Under $x_j=e^{2\pi i z_j}$, the shift $z_i\mapsto z_i+\eta$ sends $x_i\mapsto e^{2\pi i\eta}x_i$, so $p_i$ becomes $\Gamma_{i}$ with $q=e^{2\pi i\eta}$. Consequently $D^1_N\longrightarrow \left(\frac{\pi e^{\pi \imath\hbar}}{\sin(\pi\hbar)}\right)^{N-1}\mathcal M^1_N$. For the auxiliary operator, each summand tends to
$$
\begin{aligned}
&\phi(z_{N+1}-z_N)\phi(z_N-z_i)+\phi(z_N-z_{N+1})\phi(z_{N+1}-z_i)-\phi(z_N-z_i)\phi(z_{N+1}-z_i)\\
&\longrightarrow
\left(\frac{\pi e^{\pi \imath\hbar}}{\sin(\pi\hbar)}\right)^2\Big[
\frac{t x_N-x_{N+1}}{x_N-x_{N+1}}\frac{t x_i-x_N}{x_i-x_N}
+\frac{t x_{N+1}-x_N}{x_{N+1}-x_N}\frac{t x_i-x_{N+1}}{x_i-x_{N+1}}
-\frac{t x_i-x_N}{x_i-x_N}\frac{t x_i-x_{N+1}}{x_i-x_{N+1}}
\Big]\\
&\qquad = \left(\frac{\pi e^{\pi \imath\hbar}}{\sin(\pi\hbar)}\right)^2 t.
\end{aligned}
$$
Hence
$$
\begin{aligned}
\mathcal E_N
&\longrightarrow \left(\frac{\pi e^{\pi \imath\hbar}}{\sin(\pi\hbar)}\right)^{2N}
\,t\,\mathcal M^1_{N-1}.
\end{aligned}
$$
Dividing the limit of the recurrence by $\displaystyle \left(\frac{\pi e^{\pi \imath\hbar}}{\sin(\pi\hbar)}\right)^{2N}$ yields exactly the type-$A_N$ Macdonald operator rank recurrence relations \eqref{eq: Level 1 type A recurrence}.
\section{Discussion}

    The operators \eqref{eq:type D_1}-\eqref{eq:type_CN_macdo} form a commuting family of difference operators and have an intimate relationship with the representation theory of Double Affine Hecke Algebras (DAHA), see for example, \cite{Cherednik1995,Cherednik1997,Cherednik2005,Sun16}. Our work here does not seem to suggest any direct connection to DAHA literature, even in the case of type $A_{N}$ Macdonald operators. This is partly because in DAHA setting one typically fixes the number of variables that the operators act on. The procedure of removing variables resembles the notion of \emph{parabolic subgroups} of the underlying Weyl group, introduced in the earlier work of \cite{KazhdanLusztig1987,Deodhar1987,Lusztig1989} and also mentioned in the context of DAHA in \cite{Cherednik94,Cherednik2005}. However, we were not able to establish any algebraic connection between our auxiliary operators and the parabolic subgroups.

Let us note that in the case of Ruijsenaars operators, the last term is not itself a Ruijsenaars operator of lower rank, but rather involves a more complicated combination of operators. Nevertheless, in the trigonometric limit, this auxiliary operator reduces to a lower-rank Macdonald operator and matches the recurrences for the type-$A$ Macdonald operators. We suspect that this holds for higher order Ruijsenaars operators as well, where there exists a similar reduction to lower-rank Macdonald operators in the trigonometric limit but the precise form of the auxiliary operator may be more complicated. 
    
It is also worth mentioning that the Ruijsenaars operators have  higher order generalizations introduced in \cite{FeiginSilantyev2014} with elliptic extension by \cite{HLNR22}. We notice that in \cite{HLNR22}, the four generalized higher order Ruijsenaars operators can also be written in terms of other operators with coefficients are determinants of certain matrices (see Corollary 2.2 in \cite{HLNR22}). This is somewhat analogous to our formula in the previous work in \cite{DV26}, Remark 3.8.

\vskip 2cm
\textbf{AI Statement}: The paper used Deepseek-V4.1-Flash to verify the author's ansatz for Section 2 and produced the formula for type $C_N$ Macdonald operators decomposition. The author verified all the statements and we take full responsibility for the correctness of the result.

\newpage

\begin{appendix}
    \section{Root Systems and Their Associated Macdonald Operators}
In this appendix, we follow the presentation in \cite{Macdonald2000} and adding details where needed for our paper. 
\subsection{Root Systems and Their Weyl Groups}
Fix a real vector space $V$ endowed with a symmetric bilinear form $(\cdot,\cdot)$. A \emph{root system} $\Phi$ in $V$ is a finite set of non-zero vectors (called \emph{roots}) satisfying the following conditions:
\begin{enumerate}
    \item $\Phi$ spans $V$.
    \item If $\alpha \in \Phi$, then the only scalar multiples of $\alpha$ in $\Phi$ are $\alpha$ and $-\alpha$.
    \item For any $\alpha \in \Phi$, the reflection $s_\alpha$ defined by $s_\alpha(v) = v - 2\frac{(\alpha,v)}{(\alpha,\alpha)}\alpha$ leaves $\Phi$ invariant.
    \item For any $\alpha, \beta \in \Phi$, the number $2\frac{(\alpha,\beta)}{(\alpha,\alpha)}$ is an integer.
\end{enumerate}
Define the coroot $\alpha^\vee$ corresponding to a root $\alpha \in \Phi$ by
\[
    \alpha^\vee = \frac{2\alpha}{(\alpha,\alpha)}.
\]

The \emph{weight lattice} is defined as $\Lambda: = \{\lambda \in \R^N \mid (\lambda, \alpha_i^\vee) \in \mathbb{Z}, 1\le i\le n\}$ and members of $\Lambda$ are called \emph{weights}. Similarly, define the coweight lattice $\Lambda^\vee: = \{\lambda^\vee \in \R^N \mid (\alpha_i, \lambda^\vee) \in \mathbb{Z}, 1\le i\le n\}$ and members of $\Lambda^\vee$ are called \emph{coweights}. In short, weights pair integrally with coroots, while coweights pair integrally with roots. A coweight $\lambda^\vee \in \Lambda^\vee$ is called \emph{dominant} if $(\lambda^\vee, \alpha_i) \ge 0$ for all $i$ and the dominant coweight is called \emph{minuscule} if $(\lambda^\vee, \alpha_i) \in \{0,1\}$ for all $i$. We also defined \emph{fundamental coweights} to be those vectors $\omega_a^\vee \in \Lambda^\vee$ such that $(\omega_a^\vee, \alpha_i) = \delta_{ai}$ for all $i$. We summarize all of the above data in Table \ref{table:root-weight-data}.

\begin{table}
\centering
\small
\begin{adjustbox}{max width=\textwidth}
\begin{tabular}[t]{|l|l|l|l|l|}
\hline Algebra & Fundamental weights $\omega_a$ & Fundamental coweights $\omega_a^\vee$ & simple roots $\alpha_a$ & simple coroots $\alpha^\vee_a$ \\
\hline $A_{N-1}$ & $\omega_a=\sum_{i=1}^a \hat{e}_i, a \in[1, N-1]$ & $\omega_a^\vee=\omega_a$ & $\alpha_a=e_a-e_{a+1}$ & $\alpha^\vee_a=e_a-e_{a+1}$ \\
\hline $B_N$ & $\omega_a= \begin{cases}\sum_{i=1}^a e_i, & a<N ; \\ \frac{1}{2} \sum_{i=1}^N e_i, & a=N .\end{cases}$ & $\omega^\vee_a = \sum_{i=1}^a e_i, 1\leq a\leq N$ & $\begin{aligned}
& \alpha_a=e_a-e_{a+1}, a<N \\
& \alpha_N= e_N, a=N
\end{aligned}
$ & $\begin{aligned}
& \alpha^\vee_a = e_a-e_{a+1}, a<N \\
& \alpha^\vee_N = 2 e_N, a=N
\end{aligned}
$ \\
\hline $C_N$ & $\omega_a=\sum_{i=1}^a e_i, a \in[1, N]$ & $\omega^\vee_a=\frac{1}{2}\begin{cases}\sum_{i=1}^a e_i, a \in[1, N-1]\\  \sum_{i=1}^N e_i, a=N \end{cases}$ & $\begin{aligned}
& \alpha_a=e_a-e_{a+1}, a<N \\
& \alpha_N= 2 e_N, a=N
\end{aligned}$ & $\begin{aligned}
& \alpha^\vee_a =e_a-e_{a+1}, a<N \\
& \alpha^\vee_N = e_N, a=N
\end{aligned}$ \\
\hline $D_N$ & $\begin{aligned}
& \omega_a=\sum_{i=1}^a e_i, a<N-1 \\
& \omega_{N-1}=\frac{1}{2}\left(\omega_{N-2}+e_{N-1}-e_N\right) \\
& \omega_N=\frac{1}{2}\left(\omega_{N-2}+e_{N-1}+e_N\right)
\end{aligned}$ & $\omega_a^\vee=\omega_a$ & $\begin{aligned}
& \alpha_a=e_a-e_{a+1}, a<N \\
& \alpha_N=e_{N-1}+e_N
\end{aligned}$ & $\begin{aligned}
& \alpha^\vee_a=e_a-e_{a+1}, a<N \\
& \alpha^\vee_N=e_{N-1}+e_N
\end{aligned}$ \\
\hline
\end{tabular}
\end{adjustbox}
\caption{Root and weight data for the classical Lie algebras}
\label{table:root-weight-data}
\end{table}

Given a root system $\Phi$, define the \emph{Weyl group} of $\Phi$ to the group generated by the reflections $s_\alpha$ for all $\alpha \in \Phi$. We abuse the notation and use $x_i$ for the coordinates of a vector in the standard orthonormal basis of $\R^N$ similar to those in Table \ref{table:root-weight-data}, i.e. a vector $v=\sum_{i=1}^Na_ie_i$. Then, let us derive the Weyl group for each classical root system in terms of the $a_i$ coordinates.

\textbf{Type $A_{N-1}$:}
The simple roots are $\alpha_r = e_r - e_{r+1}$ for $r=1, \ldots, N-1$. Then,
\[s_{\alpha_i}(v)=v - 2\frac{(v,\alpha_i)}{(\alpha_i,\alpha_i)}\alpha_i =v-(v_i-v_{i+1})(e_i - e_{i+1}) = (v_1, \ldots, v_{i+1}, v_i, \ldots, v_N).\]
Then, the Weyl group of type $A_{N-1}$ is isomorphic to the symmetric group $S_N$, which are all transpositions $i \mapsto i+1$.

\textbf{Type $B_N$:} The simple roots  $\alpha_r = e_r - e_{r+1}$ are similar to those of type $A_{N-1}$ for $r=1, \ldots, N-1$. The Weyl group element corresponds to $\alpha_N$ is
\[s_{\alpha_N}(v) = v - 2\frac{(v,\alpha_N)}{(\alpha_N,\alpha_N)}\alpha_N = v - 2v_N e_N = (v_1, \ldots, v_{N-1}, -v_N).\]
Then, the Weyl group of type $B_N$ is isomorphic to $S_N \ltimes (\mathbb{Z}/2\mathbb{Z})^N$, which comprises all signed permutations of the coordinates.

\textbf{Type $C_N$:} Similar to type $B_N$, with the simple root $\alpha_N = 2 e_N$ giving the Weyl group element
\[s_{\alpha_N}(v) = v - 2\frac{(v,\alpha_N)}{(\alpha_N,\alpha_N)}\alpha_N = v - 2v_N e_N = (v_1, \ldots, v_{N-1}, -v_N).\]

\textbf{Type $D_N$:} Besides transpositions, we also have the Weyl group element corresponding to the simple root $\alpha_N = e_{N-1} + e_N$:
\[s_{\alpha_N}(v) = v - 2\frac{(v,\alpha_N)}{(\alpha_N,\alpha_N)}\alpha_N = v - (v_{N-1} + v_N)(e_{N-1} + e_N) = (v_1, \ldots, v_{N-2}, -v_{N-1}, -v_N).\]
Then, the Weyl group of type $D_N$ is isomorphic to $S_N \ltimes (\mathbb{Z}/2\mathbb{Z})^{N-1}$, which are all signed permutations of the coordinates with an even number of sign changes.

\subsection{Orbit Sum Formula}
\label{subapp:Orbit Sum}
The Macdonald operators associated to these root systems are defined as the Weyl group orbit sum: 
\begin{equation}
    \label{eq: Weyl group sum}
    \mathcal E_\lambda=\sum_{\mu \in W\lambda}\prod_{\substack{\alpha \in \Phi \\ \langle \mu,\alpha\rangle>0}} \frac{1-t x^{\alpha}}{1-x^{\alpha}} \Gamma^\mu
\end{equation}
where here $x$ is now a variable and $x^\alpha=\prod_{i=1}^N x_i^{\alpha_i}$ for $\alpha=\sum_{i=1}^N a_i e_i$. Then, one can verify from Table \ref{table:root-weight-data} that every fundamental coweight of $A_{N-1}$ is minuscule, while there is only one minuscule fundamental coweight for $B_N$ and $C_N$ and three for $D_N$. We will use the notation $\Phi_{B_N}, \Phi_{C_N}$ and $\Phi_{D_N}$ to denote their root system. Let us derive the Macdonald operators corresponding to these minuscule coweights in the case of type $B_N$, $C_N$ and $D_N$. 
One can verify that from all the fundamental coweight in Table \ref{table:root-weight-data}, the only minuscule fundamental coweight is $\omega_1^\vee=e_1$ for type $B_N$, $\displaystyle \omega^\vee_N=\frac{1}{2}\sum_{i=1}^{N} e_i$ for type $C_N$, and $\omega_1^\vee=e_1$, $\displaystyle \omega_{N-1}^\vee=\frac{1}{2}\sum_{i=1}^{N-1} e_i - e_N$, $\displaystyle \omega_N^\vee=\frac{1}{2}\sum_{i=1}^{N-1} e_i + e_N$ for type $D_N$. The Weyl orbit of these minuscule coweights are: 
\begin{itemize}
\item \emph{Type $B_N$}: $W\omega_1^\vee = \{\pm e_i \mid i=1, \ldots, N\}$. Thus, the sum in \eqref{eq: Weyl group sum} runs over all $\mu = \pm e_i$ for $i=1, \ldots, N$. Then, the only root $\alpha \in \Phi_{B_N}$ such that $( \mu, \alpha ) > 0$ are $\alpha = \epsilon e_i, \epsilon e_i+e_j$ and $\epsilon e_i-e_j$ for $\mu = \epsilon e_i$. These contribute to the exactly the factors: 
\[\frac{1-t x_i^\epsilon}{1-x_i^\epsilon} \prod_{j \neq i} \frac{1-t x_i^\epsilon x_j}{1-x_i^\epsilon x_j} \frac{t x_i^\epsilon-x_j}{x_i^\epsilon-x_j}.\]
Then the operator is exactly as \eqref{eq:type_BN_macdo}.
\item \emph{Type $C_N$}: $W\omega_N^\vee = \left\{\frac{1}{2}\sum_{i=1}^{N} \epsilon_i e_i \mid \epsilon_i = \pm 1\right\}$. Thus, the sum in \eqref{eq: Weyl group sum} runs over all $\mu = \frac{1}{2}\sum_{i=1}^{N} \epsilon_i e_i$ with $\epsilon_i = \pm 1$. In the root system of type $C_N$ the only roots $\alpha \in \Phi_{C_N}$ such that $\langle \mu, \alpha \rangle > 0$ are $\alpha = 2\epsilon_i e_i$ and $\alpha = \epsilon_i e_i + \epsilon_j e_j$ for $i<j$. These contribute to the exactly the factors: 
\[\frac{1-t x_i^{2 \epsilon_i}}{1-x_i^{2 \epsilon_i}} \prod_{1 \leq i<j \leq N} \frac{1-t x_i^{\epsilon_i} x_j^{\epsilon_j}}{1-x_i^{\epsilon_i} x_j^{\epsilon_j}}\]
\item \emph{Type $D_N$}: 
The minuscule coweights are:
\[
\omega_1^\vee=e_1,\qquad
\omega_{N-1}^\vee=\frac{1}{2}\left(\sum_{i=1}^{N-1} e_i-e_N\right),\qquad
\omega_N^\vee=\frac{1}{2}\left(\sum_{i=1}^{N-1} e_i+e_N\right).
\]

For $\omega_1^\vee=e_1$, we have $W\omega_1^\vee=\{\pm e_i\mid i=1,\ldots,N\}$. Hence for $\mu=\epsilon e_i$, the roots $\alpha\in\Phi_{D_N}$ with $(\mu,\alpha)>0$ are $\alpha=\epsilon e_i\pm e_j$ with $j\neq i$, which contribute the factors
\[
\prod_{j\neq i}\left(\frac{1-tx_i^{\epsilon}x_j}{1-x_i^{\epsilon}x_j}\right)\left(\frac{1-tx_i^{\epsilon}x_j^{-1}}{1-x_i^{\epsilon}x_j^{-1}}\right).
\]

For $\omega_{N-1}^\vee=\frac{1}{2}\left(\sum_{i=1}^{N-1} e_i-e_N\right)$, we have
\[
W\omega_{N-1}^\vee=\left\{\frac{1}{2}\sum_{i=1}^{N-1}\epsilon_i e_i-e_N \middle| \epsilon_i=\pm1,\ \prod_{i=1}^{N-1}\epsilon_i=1\right\}.
\]
Writing $\mu=\frac12\sum_{k=1}^{N}\delta_k e_k$, the roots with $(\mu,\alpha)>0$ are $\alpha=\delta_i e_i+\delta_j e_j$ for $1\le i<j\le N$, and they contribute
\[
\prod_{1\le i<j\le N}\frac{1-tx_i^{\delta_i}x_j^{\delta_j}}{1-x_i^{\delta_i}x_j^{\delta_j}}.
\]
Similar factor for the case  $\omega_N^\vee=\frac{1}{2}\left(\sum_{i=1}^{N-1} e_i+e_N\right)$

\end{itemize}

\end{appendix}

\newpage
% Bibliography disabled until real citations are added.
\bibliographystyle{alpha}
\bibliography{ref.bib}

\end{document}